\documentclass[12pt, reqno]{amsart}
\usepackage{amsmath, amsthm, amscd, amsfonts, amssymb, graphicx, color}
\usepackage[bookmarksnumbered, colorlinks, plainpages]{hyperref}
\usepackage{float}
\usepackage{graphicx,epstopdf}
\usepackage{bbm}
\usepackage{mathrsfs}
\usepackage{dsfont}
\usepackage{amssymb}
\usepackage{csquotes}
\usepackage[dvips]{epsfig}
\usepackage{latexsym}
\usepackage{exscale}
\usepackage[latin1]{inputenc}

\usepackage{tikz}
\usepackage{tikz-cd}
\usetikzlibrary{matrix,arrows,decorations.pathmorphing}

\hypersetup{colorlinks=true,linkcolor=red, anchorcolor=green, citecolor=cyan, urlcolor=red, filecolor=magenta, pdftoolbar=true}

\newtheorem{theorem}{Theorem}[section]
\newtheorem{lemma}[theorem]{Lemma}
\newtheorem{proposition}[theorem]{Proposition}
\newtheorem{corollary}[theorem]{Corollary}
\theoremstyle{definition}
\newtheorem{definition}[theorem]{Definition}
\newtheorem{example}[theorem]{Example}

\theoremstyle{remark}
\newtheorem{remark}[theorem]{Remark}
\numberwithin{equation}{section}

\newcommand{\abs}[1]{\left\lvert#1\right\rvert}
\newcommand{\N}{\mathbb{N}} 
\newcommand{\F}{\mathcal{F}}
\newcommand{\Lin}{\mathcal{L}}

\begin{document}

\setcounter{page}{1}

\title{Equality in Degrees of Compactness: Schauder's Theorem and s-numbers}

\begin{center}
\author[A. G. AKSOY, D. A. THIONG ]{Asuman G\"{u}ven AKSOY, Daniel Akech Thiong}
\end{center}

\address{$^{*}$Department of Mathematics, Claremont McKenna College, 850 Columbia Avenue, Claremont, CA  91711, USA.}
\email{\textcolor[rgb]{0.00,0.00,0.84}{aaksoy@cmc.edu }}

\address{$^{1}$Department of Mathematics, Claremont Graduate University, 710 N. College Avenue, Claremont, CA  91711, USA.}
\email{\textcolor[rgb]{0.00,0.00,0.84}{daniel.akech@cgu.edu}}

\subjclass[2010]{Primary 47A16, 47B10; Secondary 47A68}

\keywords{s-numbers, approximation schemes, Schauder's theorem }


\begin{abstract} 
We investigate an extension of Schauder's theorem by studying the relationship between various $s$-numbers of an operator $T$ and its adjoint $T^*$. We have three main results. First, we present a new proof that the approximation number of $T$ and $T^*$ are equal for compact operators. Second, for non-compact, bounded linear operators from $X$ to $Y$, we obtain a relationship between certain $s$-numbers of $T$ and $T^*$ under natural conditions on $X$ and $Y$. Lastly, for non-compact operators that are compact with respect to certain approximation schemes, we prove results for comparing the degree of compactness of $T$ with that of its adjoint $T^*$.

\end{abstract} \maketitle

\section{Introduction}
In the following, we give a brief review of the background, notation, and terminology that will be relevant to this paper. 
Let $\mathcal{L}(X,Y)$ denote the normed vector space of all continuous operators from $X$ to $Y$, $X^*$ be the dual space of $X$, and $\mathcal{K}(X,Y)$ denote the collection of all compact operators from $X$ to $Y$. Denote by 
$T^{*} \in \mathcal{L}(Y^{*}, X^{*} )$ the adjoint operator of $T\in \Lin (X, Y)$. The well known theorem of Schauder states that $T \in \mathcal{K}(X,Y)$  if and only if $T^{*} \in \mathcal{K}(Y^{*},X^{*})$.  The proof of Schauder's theorem that uses Arzel$\grave{a}$-Ascoli Theorem is presented in most textbooks on functional analysis (see, e.g.,  \cite {Rudin}).  A new and simple proof which does not depend on Arzel$\grave{a}$-Ascoli can be found in \cite{Runde}. Recalling the fact that a class of operators $\mathcal{A}(X,Y)\subset \mathcal{L}(X,Y)$ is called \textit{symmetric} if $T \in \mathcal{A}(X,Y)$ implies $T^{*} \in \mathcal{A}(Y^{*},X^{*})$, we note that  Schauder's Theorem assures that the class $\mathcal{K}(X, Y)$ of compact operators between arbitrary Banach spaces $X$ and $Y$ is  a symmetric  operator ideal  in $\mathcal{L}(X, Y)$.  

In \cite{Rie} F. Riesz proved compact operators have at most countable set of eigenvalues $\lambda_{n}(T)$, which arranged in a sequence, tend to zero. This result raises the question of what are the conditions on  $T\in \Lin (X, Y)$ such that $(\lambda_{n}(T)) \in \ell_{q}$? Specifically, what is the rate of convergence to zero of the sequence $(\lambda_{n}(T))$? To answer these questions, in \cite{PieID} and \cite{Pieeig}, A. Pietsch developed $s$-numbers $s_{n}(T)$ (closely related to singular values), which characterize the degree of compactness of $T$. The concept of \textit{s-numbers}  $s_n(T)$  is introduced axiomatically in \cite{PieID},  their relationship to eigenvalues  are given in detail in \cite{Pieeig}. 

\begin{definition}
A map  which assigns to every operator $T$  a scalar sequence, is said to be a  $s$-function  if the following conditions are satisfied:
\begin{enumerate}
\item $||T|| =s_1(T) \geq s_2(T) \geq \dots \geq 0$ for $T \in \mathcal{L}(X,Y)$.
\item $s_{m+n-1} (S+T) \leq s_m(T)+s_n(T) $ for $S,T \in \mathcal{L}(X,Y)$ .
\item $s_n(RTK) \leq ||R|| s_n(T) ||K|| $ \\for  $K \in  \mathcal{L}(X_0,X)$,   $T \in \mathcal{L}(X,Y)$,  $R \in \mathcal{L}(Y,Y_o)$.
\item If rank $(T) < n$, then $s_n(T) =0$.
\item $s_n(I_n)= 1$ where $I_n$ is the identity map of $\ell_2^n$.
\end{enumerate}
We call $s_n(T)$ the n-th \textit{$s$-number} of the operator $T$. Observe that $s_n(T)$ depends on $T$ continuously since $$ |s_n(S)-s_n(T)| \leq ||S-T || .$$
\end{definition}
In \cite{PieID} it is shown that there is only one $s$-function on the class of all operators between Hilbert spaces.  For example, if we let  $T$ be a diagonal operator acting on $\ell_2$ such that $$ T(x_n)= (\lambda_n x_n), \quad\mbox{where} \,\,\, \lambda_1 \geq \lambda_2 \geq \dots \geq 0\quad \mbox{then}\,\,\, s_n(T) =\lambda_n $$
for every $s$-function. 

However for Banach spaces
 there are several possibilities  of  assigning  to every operator $T: X \to Y$ a certain sequence of numbers $\{s_n(T)\}$ 
 which characterizes the degree of approximability or compactness of $T$. The main examples of s-numbers to be used in this paper are approximation numbers, Kolmogorov numbers, Gelfand numbers  and symmetrized approximation numbers which  are all  defined below. 


First, for two arbitrary normed spaces $X$ and $Y$, we define the collection of the finite-rank operators as follows: $$\F_n(X, Y) = \{A \in \mathcal{L}(X, Y): \text{rank} (A) \leq n  \}, \quad\mbox{and} \quad  \F(X, Y) =\bigcup _{n=0}^{\infty} \F_n(X, Y) $$ which forms the smallest ideal of operators that exists. 

\begin{definition} In the following we define the s-numbers we will use.
\begin{enumerate}
\item The \textit{nth approximation number} $$a_{n}(T) = \inf\{||T - A||: A \in \F_n(X, Y)\},\quad n=0,1,\dots$$
Note that $a_{n}(T)$ provides a measure of how well T can be approximated by finite mappings whose range is at most n-dimensional.  It is clear that the sequence $\{a_n(T)\}$ is monotone decreasing and $\displaystyle \lim_{n \to \infty} a_n(T) = 0$  if and only if $T$ is the limit of finite rank operators.   It is known that the largest $s$-number is the approximation number. This is so because $a: S\to (a_n(S))$ is an $s$-function and if we consider $S\in \mathcal{L} (X,Y)$ and  if $L\in \mathcal{F}(X,Y)$ with $\text{rank}(L)< n$, then 
$$ s_n(S) \leq s_n(L)+ || S-L||=||S-L||.$$ Therefore $s_n(S) \leq a_n(S)$. See \cite{CS} or \cite {PieID}  for more details.

\item The \textit{nth  Kolmogorov diameter} of $T \in \mathcal{L}(X)$ is defined by $$\delta_{n}(T) = \inf \{||Q_{G} T||: \dim G \leq n \}$$ where the infimum is over all subspaces $G \subset X$  such that $\dim G \leq n$ and $Q_{G}$ denotes the canonical quotient map $Q_{G}: X \rightarrow X/G$. \\
\item The \textit{nth Gelfand number of }$T$, $c_{n} (T)$ is defined as: $$c_{n}(T) = \inf \{ \epsilon > 0: ||Tx|| \leq \sup_{1 \leq i \leq k} | \langle x, a_{i} \rangle |+ \epsilon ||x|| \}$$
 where $  a_{i} \in X^{*}, 1 \leq i \leq k \text { with }  k < n  $. 
It follows that an operator $T$ is compact if and only if $c_{n}(T) \to 0$ as $n \to \infty$.  


\item The \textit{nth symmetrized approximation number} $\tau_{n}(T) $ for any operator $T$  defined between arbitrary Banach spaces $X$ and $Y$ is defined as follows: $$\tau_{n}(T) = \delta_{n}(J_{Y} T)\quad \mbox 
{where} \quad  J_{Y}: Y \to \ell_{\infty} (B_{Y^{*}})$$ is an embedding map.
Note that above definition is equivalent to $$\tau_{n}(T)  = a_{n}(J_{Y}TQ_{X})$$ as well as to  $$\tau_{n}(T)  =c_{n}(TQ_{X}), $$ where $Q_{X}: \ell_{1}(B_{X}) \rightarrow X$   is a metric surjection onto $X$ given by
$Q_{X}( \xi_{x})= \sum_{B_{X}} \xi_x x$  \,\, for\,\,  $(\xi_{x}) \in \ell_1(B_{X})$ . 

\end{enumerate}
\end{definition} 
It is possible to compare various s-numbers such as $a_{n}(T), \, \delta_{n}(T),\,  c_{n}(T)$ if one imposes some mild restrictions on $X$ and $Y$. With this purpose in mind we define  well known concepts of lifting and extension properties. 

\begin{definition}    In the following we introduce two well-known important  properties of  Banach spaces. See \cite{CS} for details.
\begin{enumerate}

\item We say that a Banach space $X$ has \textit{the lifting property} if for every $T\in \mathcal{L}(X, Y/F)$ and every $\epsilon >0$ there exists an operator $S \in \mathcal{L}(X, Y)$ such that $$||S|| \leq (1 + \epsilon) ||T||\quad\mbox{and}\,\,\,  T = Q_{F}S,$$ where $F$ is a closed subspace of the Banach space $Y$ and $Q_{F}: Y \rightarrow Y/ F$ denotes the canonical projection. 
\begin{example} The Banach space $\ell_{1} (\Gamma)$ of \textit{summable number families} $\{ \lambda_{\gamma } \}_{\gamma \in \Gamma} $ over an arbitrary index set $\Gamma$, whose elements $\{ \lambda_{\gamma } \}_{\gamma \in \Gamma} $  are characterized by $\sum_{\gamma \in \Gamma} | \lambda_{\gamma }| < \infty$, has the metric lifting property. 
\end{example}

\item A Banach space $Y$ is said to have \textit{the extension property} if for each $T \in \mathcal{L}(M, Y)$ there exists an operator $S \in \mathcal{L}(X, Y)$ such that $T = SJ_{M}$ and $||T|| = ||S||$, where $M$ is a closed subspace of an arbitrary Banach space $X$ and $J_{M}: M \rightarrow Y$  is the canonical injection. 
\begin{example}The Banach space $\ell_{\infty} (\Gamma)$ of \textit{bounded number families} $\{ \lambda_{\gamma } \}_{\gamma \in \Gamma} $  over an arbitrary index set $\Gamma$ has the metric extension property. 
\end{example}
\end{enumerate}
\end{definition} 

We mention a couple of facts to illustrate the importance of lifting and extensions properties with respect to $s$-numbers. If $T$ is any map from a Banach space with metric lifting property to an arbitrary Banach space, then $a_{n}(T) = \delta_{n} (T)$ (\cite{CS}, Prop. $2.2.3)$. It is also known that every Banach space $X$ appears as a quotient space of an appropriate space $\ell_{1} (\Gamma)$ (see \cite{CS}, p.$52$).  Furthermore, If $T$ is any map from an arbitrary Banach space into a Banach space with metric extension property, then $a_{n}(T) = c_{n} (T)$ (\cite{CS}, Prop. $2.3.3$). Additionally, every Banach space $Y$ can be regarded as a subspace of an appropriate space $\ell_{\infty} (\Gamma)$ (see \cite{CS}, p.$ 60$). 

For non-compact operator $T \in \mathcal{L}(X,Y)$, we do not have too much  information about the relationship between $s_{n}(T)$ with $s_{n}(T^{*})$. In this paper,  by imposing certain natural conditions on $X$ and $Y$ we are able to obtain a relationship between  $s_{n}(T)$ with $s_{n}(T^{*})$ for certain s-numbers.  Moreover,  using a new characterization of compactness due to Runde \cite{Runde} together with the Principle of Local Reflexivity, we give a different, simpler proof of Hutton's theorem \cite{Hut} establishing that  for any compact map $T$, $$a_n(T) =a_n(T^*) \quad\mbox{for all} \,\,\,n.$$ Next we consider operators which are not compact but compact with respect to certain approximation schemes Q.  We call such operators as  Q-compact and   prove that  for any  Q-compact operator $T$, one has $\tau_{n}(T) =  \tau_{n}(T^{*})$. This result answers the question of comparing the degree of compactness for $T$ and its adjoint $T^{*}$ for   non-compact operators $T$. 
\section{ Comparing $s_n(T)$ and $s_n(T^*)$}


 Hutton in  \cite{Hut} used the Principle of Local Reflexivity  (PLR)  to prove  that for $T \in \mathcal{K}(X,Y)$  we have $$ a_{n}(T) = a_{n} (T^{*})\quad\mbox{ for all} \,\,\, n.$$ This result fails for non-compact operators. For example, if $T = I: \ell_{1} \rightarrow c_{0}$ is the canonical injection and $T^{*}: \ell_{1} \rightarrow \ell_{\infty}$  is the natural injection, then one can show  $$1 = a_{n}(T) \neq a_{n}(T^{*}) = \frac{1}{2}.$$

On the other hand by considering the ball measure of non-compactness, namely,  $$\gamma (T) := \inf \{r > 0: T(B_{X}) \subset \bigcup_{k =1}^{n} A_{k}, \,\, \displaystyle \text{max}_{1\leq k \leq n} \text{ diam } (A_{k}) < r, \, n\in \mathbb{N} \} $$ Astala  in  \cite{As} proved that if $T \in \mathcal{L}(X,Y)$, where X and Y are arbitrary Banach spaces with metric lifting and extension property, respectively, then $$\gamma (T) = \gamma(T^{*}).$$ 
Our first result is  a  different,  simpler proof  of Hutton's Theorem. We use only  the characterization of compactness by Runde \cite{Runde}, together with the Principle of Local Reflexivity.  Lindenstrass and Rosenthal \cite{LR}  discovered a principle that shows that all Banach spaces are ``locally reflexive" or said in another way,  every bidual $X^{**}$ is finitely representable in the original space $X$. The following  is a stronger version of this property called \textit{Principle of Local Reflexivity} (PLR) due to Johnson, Rosenthal and Zippin \cite{JRZ}:

\begin{definition}
Let $X$ be  a Banach space regarded  as a subspace of $X^{**}$, let $E$ and $F$  be finite dimensional subspaces of $X^{**}$ and $X^*$ respectively and  let $\epsilon >0$. Then there exist a one-to-one operator $T: E \to X$ such that
\begin{enumerate}
\item $T(x)= x$  for all $x\in X \cap E$
\item $f(Te)=e(f)$ for all $e\in E$ and $f\in F$
\item $||T|| ||T^{-1}|| < 1+\epsilon$.
\end{enumerate}
\end{definition}
 PLR is an effective tool in Banach space theory. For example  Oja and Silja in \cite{Oja}  investigated versions of the principle of local reflexivity for nets of subspaces of a Banach space and gave some applications to  duality and lifting theorems.

\begin{lemma} [Lemma $1$ in \cite{Runde}]
Let $X$ be a Banach space and let $T \in \mathcal{L}(X)$. Then $T \in \mathcal{K}(X)$ if and only if, for each $\epsilon >0$, there is a finite-dimensional subspace $F_{\epsilon}$ of $X$ such that $||Q_{F_{\epsilon}}T|| < \epsilon$, where $Q_{F_{\epsilon}}: X \rightarrow X/F_{\epsilon}$ is the canonical projection. 

\end{lemma}

\begin{theorem} 
Let $T \in \mathcal{K}(X)$. Then $a_{n} (T) = a_{n} (T^{*})$ for all $n$. 

\end{theorem} 
\begin{proof}
Since one always has $a_n(T^*)\leq a_n(T)$, if we have $a_n(T)\leq a_n(T^{**})$, then $a_n(T^{**})\leq a_n(T^{*})$ would imply $a_n(T)\leq a_n(T^{*})$. 
Thus we must verify $a_n(T)\leq a_n(T^{**})$. 
To this end, suppose $T \in \mathcal{K}(X)$, by Schauder's theorem, $T^{*}$ and $T^{**}$ are compact. Let $\epsilon > 0$, then by definition, there exists $A \in \mathcal{F}_n(X^{**})$ such that $||T^{**} - A|| < a_{n} (T^{**}) + \epsilon$. 

By Lemma 2.2,  there are finite-dimensional subspaces $E_{\epsilon}$ of $X^{**}$ and $F_{\epsilon}$ of $X^{*}$ such that $||Q_{E_{\epsilon}}T^{**}|| < \epsilon$, where $Q_{E_{\epsilon}}: X^{**} \rightarrow X^{**}/E_{\epsilon}$ and  $||Q_{F_{\epsilon}}T^{*}|| < \epsilon$, where $Q_{F_{\epsilon}}: X^{*} \rightarrow X^{*}/F_{\epsilon}$.

By the Principle of Local Reflexivity (PLR), there exists a one-to-one linear operator $S: E_{\epsilon} \rightarrow X$ such that $||S||||S^{-1}|| < 1 + \epsilon$, $y^{*}(Sx^{**}) = x^{**}(y^{*})$ for all $x^{**} \in E_{\epsilon}$ and all $y^{*} \in F_{\epsilon}$, and $S_{| E_{\epsilon} \cap X} = I$.

Let $J: X \to X^{**}$ be the canonical map. By the Hahn-Banach theorem, since $E_{\epsilon} $ is a subspace of $X^{**}$, $S: E_{\epsilon} \rightarrow X$ can be extended to a linear operator $\overline{S}: X^{**} \rightarrow X$. 

We now have $T \in \mathcal{L}(X)$ and $\overline{S}AJ \in \mathcal{L}(X)$ and $\text{rank }(\overline{S}AJ ) = \text{rank} (A) < n$, and therefore $$a_{n} (T) \leq ||T - \overline{S}AJ ||.$$

To get an upper bound for $||T- \overline{S}AJ||$  we estimate $||Tx - \overline{S}AJ(x)||$ for $x \in B_{X}$ using an appropriate element $z_{j}$ of the covering of the set $T(B_{X})$.

Indeed, the compactness of $T$ implies that $T(B_{X})$ is relatively compact so that one can extract a finite-dimensional subset $Y_{\epsilon} \subset T(B_{X}) \subset X$ and let $z_{j} = Tx_{j}$ be the n elements forming a basis. 

Let $x \in B_{X}$. Then we have 

  \begin{flalign*} 
    a_{n} (T) & \leq |Tx - \overline{S}AJ(x)||  \leq ||Tx - z_{j} || + ||z_{j} - \overline{S}AJ(x)||&\\ 
                           &\leq \epsilon + ||z_{j} - \overline{S}AJ(x)|| = \epsilon + ||\overline{ S} z_{j} -  \overline{S}AJ(x)|| &\\
                           & \leq \epsilon + (1 + \epsilon) ||z_{j} - AJ(x) || <  \epsilon + (1 + \epsilon) (a_{n} (T^{*}) + \epsilon)
     \end{flalign*} 
since 
 \begin{flalign*} 
||z_{j} - AJ(x) || &= ||Jz_{j} - AJ(x)|| \leq ||Jz_{j}  - JTx|| + ||JTx -  AJ(x)||&\\
                       & \leq \epsilon + ||JTx - AJx||
                          = \epsilon + ||T^{**} Jx - AJx|| \leq ||T^{**} - A||&\\ 
                         & < a_{n} (T^{*}) + \epsilon.&
  \end{flalign*} 

It follows that  $a_{n} (T) \leq a_{n} (T^{**})$, as  promised. 
\end{proof} 

\begin{theorem} \label{kol}
If $T \in \mathcal{L}(X,Y)$, where X and Y are arbitrary Banach spaces with metric lifting and extension property, respectively, then $\delta_{n}(T^{*}) = \delta_{n}(T)$ for all $n$. 
\end{theorem} 
\begin{proof}
It is  known that if $T \in \mathcal{L}(X,Y)$, where X and Y are arbitrary Banach spaces, then $\delta_{n}(T^{*}) = c_{n}(T)$ ( \cite{CS}, Prop. $2.5.5$).
 We also know that  if $T \in \mathcal{L}(X,Y)$, where X and Y are arbitrary Banach spaces with metric lifting and extension property, respectively, then $\delta_{n} (T) = a_{n}(T) = c_{n} (T)$. 
Hence,  $$\delta_{n}(T^{*}) = c_{n}(T) = a_{n}(T) = \delta_{n}(T).\qedhere$$ 
\end{proof}
\begin{remark}
As stated before,  Astala in \cite{As} proved that if $T \in \mathcal{L}(X,Y)$, where X and Y are arbitrary Banach spaces with metric lifting and extension property, respectively, then $\gamma (T) = \gamma (T^{*})$, where $\gamma (T)$ denotes the measure of non-compactness of $T$. In \cite{Ak-Al}, it is shown that $\displaystyle \lim_{n \to \infty} \delta_{n}(T) = \gamma (T)$. This relationship between Kolmogorov diameters and the measure of non-compactness together with Theorem \ref{kol} provide an alternative proof for the result of Astala. 
\end{remark}
\begin{theorem}
If $T \in \mathcal{K}(X,Y)$, where X and Y are arbitrary Banach spaces with metric lifting and extension property, respectively, then $c_{n}(T^{*}) = c_{n}(T)$ for all $n$. 
\end{theorem} 
\begin{proof}
If $T \in \mathcal{K}(X,Y)$, then it is known that $\delta_{n}(T) = c_{n}(T^{*})$ (\cite{CS}, Prop. $2.5.6$). If X and Y are Banach spaces with metric lifting and extension property, respectively, then we  also have $\delta_{n} (T) = a_{n}(T) = c_{n} (T)$. 
Thus, $c_{n}(T^{*}) = c_{n}(T)$ for all $n$. 
\end{proof}
\begin{remark}
 In \cite{MS}
it is shown that if $X$ has the lifting property, then $X^{*}$ has the extension property. However, if $Y$ has the extension property, then $Y^{*}$ has the lifting property if and only if $Y$ is finite-dimensional. 
Therefore one can observe that if $X$ has the lifting property and $Y$ is finite-dimensional with the extension property, then  $Y^{*}$ has the lifting property and $X^{*}$ has the extension property, so that we have $$\delta_{n} (T^{*}) = a_{n}(T^{*}) = c_{n} (T^{*}).$$

\end{remark}


\section{Compactness with Approximation schemes} 
Approximation schemes were introduced in Banach space theory by Butzer and Scherer in $1968$ \cite{But} and independently by Y. Brudnyi and N. Kruglyak under the name of ``approximation families''  in \cite{BK}. They were
popularized by Pietsch in his 1981 paper \cite{Pi}, for  later developments we refer the reader to  \cite{ Ak-Al, AA,  AL}. The following definition is due to Aksoy and generalizes the classical concept of approximation scheme in a way that allows using families of subsets of $X$  instead of elements of $X$, which is useful when we deal  with n-widths.

\begin{definition}[Generalized Approximation Scheme]
Let $X$ be a Banach space.  For each $n\in \mathbb{N}$, let $Q_n=Q_n(X)$ be a family of subsets of $X$ satisfying the following conditions:
\begin{itemize}
\item[$(GA1)$] $\{0\}=Q_0\subset Q_1\subset \cdots \subset Q_n \subset \dots$.
\item[$(GA2)$]  $\lambda Q_n \subset Q_n$ for all $n \in N$ and all scalars $\lambda$.
\item[$(GA3)$]  $Q_n+Q_m \subseteq Q_{n+m}$ for every $n,m \in N$.
\end{itemize}
Then $Q(X)=(Q_n(X))_{n \in N}$ is called a \emph{generalized approximation scheme} on $X$.  We shall simply use $Q_n$ to denote $Q_n(X)$ if the context is clear.
\end{definition}
 We use here the term ``generalized'' because the elements of $Q_n$ may be subsets of $X$. Let us now give a few important examples of generalized approximation schemes.
\begin{example}

\
\begin{enumerate}
\item  $Q_n=$ the set of all at-most-$n$-dimensional subspaces of any given Banach space $X$.
\item Let $E$ be a Banach space and $X=L(E)$; let $Q_n=N_n(E)$, where $N_n(E)=$ the set of all $n$-nuclear maps on $E$ \cite {PieID}.
\item Let $a^k=(a_n)^{1+\frac{1}{k}},$ where $(a_n)$ is a nuclear exponent sequence. Then {$Q_n$}  on $X=L(E)$ can be defined as the set of all $\Lambda_\infty (a^k)$-nuclear maps on $E$ \cite{Dubinsky_Ram}.
\end{enumerate}
\end{example}
\begin{definition}[Generalized Kolmogorov Number]
Let $B_X$ be the closed unit ball of $X$,  $Q= Q(X)=(Q_n(X))_{n \in N}$ be a \emph{generalized approximation scheme} on $X$,  and $D$ be a bounded subset of $X$.  Then the $n^{\text{th}}$ \emph{generalized Kolmogorov number} $\delta_n(D;Q)$ of $D$ with respect to $Q$ is defined by
\begin{equation}
\label{GenKolmogorovNumber}
\delta_n(D;Q)=\inf\{r>0:D \subset rB_X+A \text{ for some }A \in Q_n(X)\}.
\end{equation}
Assume that $Y$ is a Banach space and $T \in \mathcal{L}(Y,X)$. The $n^{\text{th}}$ Kolmogorov number $\delta_n(T;Q)$ of $T$ is defined as $\delta_n(T(B_Y);Q)$.
\end{definition}
It follows that $\delta_n(T;Q)$ forms a non-increasing sequence of non-negative numbers:
\begin{equation}
\|T\|=\delta_0(T;Q)\geq \delta_1(T;Q)\geq \cdots \geq \delta_n(T;Q)\geq 0.
\end{equation}
We are now able to introduce $Q$-compact sets and operators:

\begin{definition}[$Q$-compact set]
Let $D$ be a bounded subset of $X$. We say that $D$ is $Q$-\emph{compact} if $\displaystyle\lim_n \delta_n(D;Q)=0$.
\end{definition}

\begin{definition}[$Q$-compact map] We say that $T\in L(Y,X)$ is a $Q$-\emph{compact map} if  $T(B_Y)$   is a  $Q$-compact set,
$$\displaystyle\lim_n \delta_n(T;Q)=0 .$$
\end{definition}
 $Q$-compact maps  are a genuine generalization of compact maps since there are examples of  $Q$-compact maps which are not compact in the usual sense. 

 In the following we present two examples of Q-compact maps which are not compact. First of this examples is known (see \cite{Ak-Al}) and it involves a projection  $P: L_p[0,1] \to R_p$ where $R_p$ denotes the closure of the span of the space of Rademacher functions.  Second example is new and illustrates the fact that if  $B_w$ is a weighted backward shift on $c_0(\N)$ with $w=(w_n)_n$ a bounded sequence not converging to 0, then  $B_w$ is $Q$-compact operator which is not compact.

\begin{example} 
Let $\{r_n(t)\}$ be the space spanned by the Rademacher functions.  It can be seen from the Khinchin inequality \cite{Lin} that
\begin{equation}
\ell_2 \approx  \{r_n(t)\}\subset L_p[0,1] \text{ for all }1\leq p \leq \infty.
\end{equation}
We define an approximation scheme $A_n$ on $L_p[0,1]$ as follows:
\begin{equation}
 A_n=L_{p+\frac{1}{n}}.
\end{equation}
$L_{p+\frac{1}{n}}\subset L_{p+\frac{1}{n+1}}$ gives us $A_n\subset A_{n+1}$. for $n=1,2,\dots,$ and it is easily seen that $A_n+A_m \subset A_{n+m}$ for $n,m=1,2,\dots,$ and that $\lambda A_n \subset A_n$ for all $\lambda$.  Thus $\{A_n\}$ is an approximation scheme.

It can be shown  that for $p\geq 2$ the projection $P: L_p[0,1] \to R_p$ is a non-compact  $Q$-compact map,
where $R_p$ denotes the closure of the span of $\{r_n(t)\}$ in $L_p[0,1]$.  (See \cite{Ak-Al} for details)

\end{example}


Next we give another example is an $Q$-operator which is not compact .


\begin{example}
Consider the weighted  backward shift $$ B(x_1,x_2, x_3, \dots)= (w_2x_2, w_3x_3, w_4 x_4, \dots)$$ where $w=(w_n)_n$ is a sequence of non-zero scalars called a \emph{weight sequence}.
Any weighted shift is a linear operator and is bounded if and only if $w$ is a bounded sequence.

\end{example}

Let $w=(w_n)_n$ be a bounded sequence of positive real numbers. The unilateral weighted shift on $c_0(\N)$ is defined by $$B_w(e_1)=0 \quad \mbox{ and } \quad B_w(e_n)=w_ne_{n-1}\quad \mbox{ for all } \quad n\ge 2.$$ 
\begin{proposition}
 Suppose  the approximation scheme $Q=(A_n)_{n=1}^{\infty}$ of $c_0(\N)$ is  defined as  $A_n=\ell_n(\N)$   for  all $n$. Then any bounded weighted shift on $c_0$ is $Q$-compact
\end{proposition}

\begin{proof}
 Let $B_w$ be any bounded and linear weighted shift on $c_0$, then $w=(w_n)_n$ is a bounded weight. Let $m\geq 1$. Consider,
\begin{align*} 
\delta_m(B_w (U_{c_0}), (A_n)_n)&=\inf\{r>0: B_w(U_{c_0})\subseteq r U_{c_0}+\ell_m\}\\
&= \inf\{r>0: \forall x\in U_{c_0}, \exists y\in U_{c_0}, \exists z\in \ell_m\mbox{  with  } B_w(x)=ry+z\}.
\end{align*}
Let $x=(x_n)_{n\geq 1}\in U_{c_0}$. Let us define $y=(y_n)_{n\geq 1}\in U_{c_0}$ and $z=(z_n)_{n\geq 1}\in \ell_1\subseteq \ell_m$ such that $B_w(x)=\frac{1}{2^m}y+z$. 

Let $A:=\{n\geq 1: 2^m\abs{x_nw_n}>1\}$. The set $A$ is finite, otherwise $(w_n)_n$ is unbounded. Set, 
\[   \left\{
\begin{array}{ll}
      x_nw_n=z_{n-1} \\
      y_{n-1}=0, & \quad \forall n\in A.\\
     \end{array} 
\right. \]
Observe that $(w_nx_n)_{n\in \N\setminus A}\in c_0$, hence there exists a subsequence $(n_k)_k$ such that $\sum_{k=1}^\infty\abs{w_{n_k}x_{n_k}}<\infty$. Set,
\[   \left\{
\begin{array}{ll}
      x_{n_k}w_{n_k}=z_{n_k-1} \\
      y_{n_k-1}=0, & \quad \forall k\geq 1.\\
     \end{array} 
\right. \]
Finally, set
\[   \left\{
\begin{array}{ll}
      2^mx_nw_n=y_{n-1} \\
      z_{n-1}=0, & \quad \forall n\in \N\setminus \{(n_k)_k\cup A\}.\\
     \end{array} 
\right. \]
 Hence, $x_nw_n=\frac{1}{2^m}y_{n-1}+z_{n-1}$, for all $n\geq 2$. In other words, $B_w(x)=\frac{1}{2^m}y+z$. Note that $y\in U_{c_0}$ and $z\in \ell_1\subset \ell_m$. In conclusion, $\delta_m(B_w( U_{c_0}),(A_n)_n)\leq \frac{1}{2^m}$. As $m$ goes to $\infty$, we obtain that $\delta_m(B_w (U_{c_0}), \ (A_n)_n)$ goes to $0$ and $B_w$ is $Q$-compact.
\end{proof}

It is well-known that $B_w$ is compact if and only if $w=(w_n)_n$ is a null-sequence.

\begin{corollary}
Let  $B_w$ be a weighted backward shift on $c_0(\N)$ with $w=(w_n)_n$ a bounded sequence not converging to 0.  Consider the approximation schemes on  $c_0(\N)$ as $Q=(A_n)_{n=1}^\infty $ with  $A_n = \ell_n(\mathbb{N})$  for all $n$. Then, $B_w$ is a non-compact $Q$-compact operator.
\end{corollary}

 Our next objective here is to ascertain whether or not Schauder's type of theorem is true for $Q$-compact maps. For this purpose we use  symmetrized approximation numbers of $T$. For our needs, we choose the closed unit ball $B_{Z}$ of the Banach space $Z$ as an index set $\Gamma$.  Our proof of the Schauder's theorem for Q-compact operators will depend on the fact that $\ell_{1} (B_{Z})$ has the lifting property and $\ell_{\infty} (B_{Z})$ has the extension property. First we recall the following proposition.

 \begin{proposition} [Refined version of Schauder's theorem \cite{CS}, p.$ 84$]
 An operator $T$ between arbitrary Banach spaces $X$ and $Y$ is compact if and only if $$\lim_{n\to \infty} \tau_{n}(T)  = 0 $$ and moreover, $$\tau_{n}(T) = \tau_{n}(T^{*}).$$

 \end{proposition}

 Motivated by this, we give the definition of Q-compact operators using the symmetrized approximation numbers. 
  
  \begin{definition} 
  We say $T$ is Q-symmetric compact if and only if $$\lim_{n \to \infty} \tau_{n} (T, Q) = 0.$$
   \end{definition} 
   \begin{remark} \label{rem1}
   
   We need the following simple facts for our proof, for details we refer the reader to \cite{CS} Prop. $2.5.4-6$.

    \begin{enumerate}
 
 \item[a)]  Recall that  $\tau_{n}(T, Q) = c_{n} (TQ_{X}, Q)$, where $Q_{X}: \ell_{1}(B_{X}) \rightarrow X$. 
 
 \item [b)] We will also abbreviate the canonical embedding $$K_{\ell_{1}(B_{Y^{*}})}: \ell_{1} (B_{Y^{*}}) \rightarrow \ell_{\infty} (B_{Y^{*}})^{*} $$ by $K$ so that $Q_{Y^{*}} = J^{*}_{Y}K$. 
 
 \item[ c)] Denote by $P_{0}: \ell_{\infty} (B_{X^{**}} ) \rightarrow \ell_{\infty} (B_{X})$ the operator which restricts any bounded function on $B_{X^{**}}$ to the subset $K_{X}(B_{X}) \subset B_{X^{**}} $ so that $Q^{*}_{X} = P_{0} J_{X^{*}}$. 
 
 \item [d)]  The relations (b) and (c) are crucial facts for the estimates of $\delta_{n}(T^{*}, Q^{*})$ and $c_{n}(T^{*}, Q^{*})$. In particular, we have $c_{n}(T^{*}, Q^{*}) \leq \delta_{n}(T, Q)$. 
\end{enumerate} 
 \end{remark}

  We now state and prove 
   the following theorem which states that  the degree of Q-compactness of $T$ and $T^*$ is the same in so far as it is measured by the symmetrized approximation numbers $\tau_{n}$. 
  
  \begin{theorem} [Schauder's theorem for Q-compact operators]
Let $T\in \mathcal{L}(X,Y)$ with $X, Y$ are arbitrary Banach spaces, and let $Q=(Q_n(X))$ be a generalized approximation scheme on $X$. Then 
 $$\tau_{n} (T^{*}, Q^{*})  = \tau_{n} (T, Q) $$
 for all $n$. 
  \end{theorem} 
  
  \begin{proof}
  
 Let us show that  $\tau_{n} (T^{*}, Q^{*})  = \tau_{n} (T, Q)$. By Remark \ref{rem1} parts  (a) and (b) we have the following estimates: 
 \begin{flalign*}
  \tau_{n} (T^{*}, Q^{*})  &= c_{n}(T^{*}Q_{Y^{*}}, Q^{*}) = c_{n}(T^{*}J_{Y}^{*} K, Q^{*}) & \\
 &\leq c_{n}((J_{Y}T)^{*}, Q^{*}) \leq \delta_{n} (J_{Y}T, Q) = t_{n}(T, Q) &
   \end{flalign*} 
   Conversely, we have by using  Remark \ref{rem1} parts  (c) and (d): 
   \begin{flalign*}
   t_{n}(T, Q) &= c_{n}(TQ_{X}, Q) = \delta_{n}(TQ_{X})^{*}, Q^{*}) &\\
    &= \delta_{n}(Q_{X}^{*}T^{*}, Q^{*} )  = \delta_{n}(P_{0} J_{X^{*}} T^{*}, Q^{*})&\\ 
    &\leq \delta_{n}(J_{X^{*}}T^{*}, Q^{*})  = t_{n} (T^{*}, Q^{*}). &
    \end{flalign*} 
  \end{proof} 
Next we define approximation numbers with respect to a given scheme as follows:
 \begin{definition}
Given an approximation scheme $\{Q_n\}$  on $X$ and $ T \in \mathcal{L}(X)$, the n-th approximation number  $a_n(T, Q)$ with respect to this approximation scheme is defined as:
$$ a_n(T, Q)=\inf \{ ||T- B|| : \,\, B\in \mathcal{L}(X), \,
\,\, B(X) \subseteq  Q_n\}$$

\end{definition}

Let $X^*$ and $X^{**}$ be the dual and second dual of $X$.  Note that if we let $J: X \rightarrow X^{**}$ be the canonical injection and let $(X, Q_{n})$ be an approximation scheme, then $(X^{**}, J(Q_{n}))$ is an approximation scheme.

Let $\{Q_n\}$ and $\{Q^{**}_n\}:=  \{ J(Q_{n}) \}$  denote the subsets of $X$ and  $X^{**}$ respectively. 

\begin{definition}
We say $(X, Q_n)$ has the \textit{Extended Local Reflexivity Property}  (ELRP) if for each countable subset $C$ of   $X^{**}$ , for each $F\in Q_n^{**}$ for some $n$  and each $\epsilon>0$, there exists a continuous linear map 
$$P: \mbox{span} (F \cup C) \to X\quad\mbox{such that}\,\, $$
 \begin{enumerate}
\item $||P|| \leq 1+\epsilon$
\item $P \restriction_{C\cap X}= I (Identity)$
\end{enumerate}
\end{definition}
Note that ELRP is an analogue of local reflexivity principle which is possessed by all Banach spaces. 

\begin{theorem} 
Suppose $(X, Q_n)$  has ELRP and $T\in \mathcal{L}(X)$ has  separable range. Then for each $n$ we have  $a_n(T, Q)= a_n(T^*, Q^*)$.
\end{theorem}

\begin{proof}
Since one always have $a_n(T^*, Q^*)\leq a_n(T, Q)$ we only need to verify $a_n(T, Q)\leq a_n(T^{**}, Q^{**})$. Let $J: X \to X^{**}$ be the canonical map and $U_X$ be the unit ball of $X.$ Given $\epsilon> 0$, choose $B\in \mathcal{L}(X^{**})$ such that $B(X^{**}) \in Q^{**}_n$ and
$$ ||B-T^{**}|| < \epsilon + a_n(T^{**}, Q^{**}_n).$$
Let $\{z_j\}$ be a countable dense set in $T(X)$, thus $Tx_j=z_j$ where $x_j\in X$. Consider  the set $$ K= \mbox{span} \{ (JTx_j)_1^{\infty} \cup B(X^{**})\}$$  applying  ELRP  of $X$ we obtain a map
$$ P : K \to X\,\,\mbox{such that} \,\,\,||P|| \leq 1+\epsilon \,\, \,\, \mbox{and}\,\,\, P \restriction_{(JTx_j)_1^{\infty} \cap X} =I$$

For $x\in U_X$, consider 
\begin{flalign*}
||Tx-PBJx|| &\leq ||Tx-z_j|| +||z_j- PBJx|| &\\
                   &\leq  \epsilon + || PJTx_j-PBJx||&\\
                   &\leq   \epsilon +(1+\epsilon) ||JTx_j- BJx|| &\\
                   &\leq  \epsilon +(1+\epsilon)[  ||JTx_j-JTx|| + || JTx-BJx||] &\\
                   &\leq  \epsilon +(1+\epsilon) [a_n(T^{**}, Q^{**}_n) +2 \epsilon ]
  \end{flalign*} 
  
  and thus $$ a_n(T, Q) \leq a_n(T^{**}, Q^{**}_n).\qedhere$$


\end{proof}


 \bibliographystyle{amsplain}

\end{document}